\documentclass[12pt]{article}
\usepackage{fullpage}
\usepackage{amsmath,amssymb,amsthm,enumerate}
\usepackage{mathrsfs}
\usepackage{txfonts}
\usepackage[dvips]{graphicx,color} %use for latex->dvips
\usepackage{xcolor}
\definecolor{dred}{rgb}{0.85,0.00,0.00}
\definecolor{dgreen}{rgb}{0.00,0.49,0.00}
\definecolor{dblue}{rgb}{0,0.08,1.00}

\newtheorem{theorem}{Theorem}[section]
\newtheorem{lemma}[theorem]{Lemma}
\newtheorem{proposition}[theorem]{Proposition}

\newtheorem{remark}[theorem]{Remark}

\newtheorem{corollary}[theorem]{Corollary}

\newtheorem{observation}[theorem]{Observation}

\def\RR{{\mathbb R}}

\def\NN{{\mathbb N}}
\newcommand\sV{\mathbf{V}}

\newcommand\sP{\mathscr{P}}

\newcommand\calV{{\mathbf{V}}}

\def\sH{{H}}

\def\d{\mathrm{d}}
\def\eps{\varepsilon}

\title{Unrestricted iterations of relaxed projections in Hilbert space: Regularity, absolute convergence, and statistics of displacements}
\author{C.~Sinan G\"unt\"urk\footnote{Courant Institute, NYU, \tt{gunturk@cims.nyu.edu}.}~ and  Nguyen T. Thao\footnote{City College, CUNY, \tt{tnguyen@ccny.cuny.edu}.}}
\date{January 22, 2019}

\begin{document}
\addtolength{\baselineskip}{2pt} % Double-spaced

\maketitle

\begin{abstract}
Given a finite collection $\sV:=(V_1,\dots,V_N)$ of closed linear subspaces of a real Hilbert space $\sH$, let $P_i$ denote the orthogonal projection operator onto $V_i$ and $P_{i,\lambda}:= (1-\lambda)I + \lambda P_i$ denote its relaxation with parameter $\lambda \in [0,2]$, $i=1,\dots,N$.
Under a mild regularity assumption on $\sV$ known as ``innate regularity'' (which, for example, is always satisfied if each $V_i$ has finite dimension or codimension), we show that all trajectories $(x_n)_{0}^\infty$ resulting from the iteration $x_{n+1} := P_{i_n,\lambda_n}(x_n)$,
where the $i_n$ and the $\lambda_n$ are unrestricted other than the assumption that $\{\lambda_n : n \in \NN\} \subset [\eta,2{-}\eta]$ for some $\eta \in (0,1]$,
possess uniformly bounded displacement moments of arbitrarily small orders. In particular, we show that
$$ \sum_{n=0}^\infty \|x_{n+1} - x_n \|^\gamma \leq C \|x_0\|^\gamma ~\mbox{ for all }~ \gamma > 0,$$
where $C:=C(\sV,\eta,\gamma)<\infty$. This result strengthens prior results on norm convergence of these trajectories, known to hold under the same regularity assumption. For example, with $\gamma=1$, it follows that the displacements series $\sum (x_{n+1}-x_n)$ converges absolutely in $\sH$.

Quantifying the constant $C(\sV,\eta,\gamma)$,
we also derive an effective bound on the distribution function of the norms of the displacements (normalized by the norm of the initial condition) which yields a root-exponential type decay bound on their decreasing rearrangement, again uniformly for all trajectories.
\end{abstract}

\section{Introduction}

Starting with the Kaczmarz method \cite{Kaczmarz37} and its many variations that have followed, projection algorithms have been employed extensively in convex feasibility problems, in particular linear inverse problems. The literature is highly mature with excellent texts and review articles; see, for example, \cite{Bauschke96,combettes1996convex,Deutsch01,cegielski2012iterative,bauschke2017convex}. 

Consider a real Hilbert space $\sH$ and a finite collection $\sV:=(V_1,\dots,V_N)$ of closed linear subspaces. For each $i\in [N]:=\{1,\dots,N\}$, let $P_i:\sH \to V_i$ be the orthogonal projection operator onto $V_i$, and for each $\lambda\in[0,2]$, let $P_{i,\lambda}:\sH \to \sH$ be its relaxation defined by 
\begin{equation}\label{relaxedPi}
P_{i,\lambda} (x):= (1-\lambda)x + \lambda P_i (x), ~~x \in \sH. 
\end{equation}
%Note that $P_{i,\lambda}(x)$ lies on the line segment joining $x$ to its mirror image relative to $V_i$ given by $2P_i(x) - x$. In particular, $P_{i,1}(x)=P_i(x)$.
We will be concerned with iterations of relaxed projections chosen arbitrarily from the collection
\begin{equation}\label{P-collection}
\sP:=\sP(\sV, \eta) := \Big \{P_{i,\lambda} : 1\leq i\leq N, \lambda \in [\eta,2-\eta] \Big \},~~ 0< \eta \leq 1. 
\end{equation}
Specifically, for each sequence $(P_{i_n,\lambda_n})_0^\infty$ in 
$\sP$ and starting point $x_0 \in \sH$, we define a trajectory $(x_n)_0^\infty$ in $\sH$ via the iteration
\begin{equation}\label{alg}
x_{n+1}:=P_{i_n,\lambda_n}(x_n),\quad n\geq 0.
\end{equation}

The $i_n$ define the so-called ``control sequence'' of the algorithm, and the $\lambda_n$ are called relaxation coefficients. In practice the control sequence may be periodic (cyclic), quasi-periodic, stochastic, or greedily determined based on some criterion, such as maximization of $\|x_n - P_i(x_n)\|$, but there has also been significant interest in unrestricted (arbitrary) control sequences (also called {\em random} or {\em chaotic control}), which is the setting of this paper.

The best known special case of \eqref{alg} involves alternating between two subspaces $V_1$ and $V_2$, with no relaxation (i.e., $\lambda_n = 1$ for all $n$). In this case, von Neumann's celebrated theorem \cite{von1950functional} says that $x_n$ converges (in norm) to the orthogonal projection of $x_0$ onto $V_1 \cap V_2$. This was extended to general $N$ in \cite{halperin1962product} for cyclic control, and later in \cite{sakai1995strong} for quasi-periodic control. 

For unrestricted iterations the situation is more complicated. In \cite{prager1960uber} norm convergence was shown to hold in finite dimensional spaces. (It was generalized in \cite{aharoni1989block} to include relaxation and convex combinations of projections.) In general Hilbert spaces, weak convergence was shown in  \cite{amemiya1965convergence} and norm convergence was proposed. This question remained unresolved for a long time, and was only answered recently, in the negative: One can find systems $\sV=(V_1,V_2,V_3)$ such that for all nonzero initial points $x_0$, norm convergence fails for some control sequences; see \cite{kopecka2014product,kopecka2017strange}.

Nevertheless, norm convergence has been shown to hold in general Hilbert spaces under mild regularity assumptions on $\sV$ (also called angle criteria); see e.g. \cite{bauschke1995norm, bauschke2001projection, pustylnik2012convergence, oppenheim2018angle}. 
In this paper, we will work with the assumption of {\em innate regularity} which was introduced in \cite{bauschke1995norm}. This concept is defined for general convex subsets, but for linear subspaces it reduces to a rather simple form: A collection $\sV=(V_1,\dots,V_N)$ is innately regular if and only if the complementary angle between $\bigcap_{i\in I}V_i$ and $\bigcap_{i\in J}V_i$ is nonzero for all subsets $I,J\subset[N]$. As a special but important case, any $\sV$ for which each $V_i$ is either finite dimensional or finite codimensional is innately regular. (For these facts, see Section \ref{angular_characterization}.)

Under the assumption of innate regularity, \cite{bauschke1995norm} showed norm convergence of unrestricted iterations of relaxed projections. In a sense, this is the best possible kind of result we can have because unlike cyclic control (or its variants where indices appear with some frequency), it is not possible to obtain any effective convergence rate guarantee for unrestricted iterations once $N\geq 3$ (even in finite dimensions), because one can adversarially slow down the speed of convergence by introducing arbitrarily long gaps for any chosen index $i$ while cycling through the remaining indices.

Nevertheless, there is still room for qualitative improvements. We show in this paper that the displacements (increments) of the resulting trajectories have bounded moments of all orders.  Our main result is the following:

\begin{theorem}\label{theo1}
Let $\calV=(V_1,\dots,V_N)$ be an innately regular collection of closed linear subspaces in a real Hilbert space $\sH$, $\eta\in(0,1]$. Let $\sP:=\sP(\calV,\eta)$ be defined as in \eqref{P-collection} and $\gamma > 0$ be arbitrary. There exists a constant $C=C(\calV, \eta,\gamma)<\infty$ such that for all $x_0\in\sH$ and all sequences of relaxed projections $(P_{i_n,\lambda_n})_0^\infty$ in 
$\sP$, the trajectory $(x_n)_0^\infty$ defined by \eqref{alg} satisfies 
$$\sum_{n=0}^\infty\|x_{n+1}-x_{n}\|^\gamma \leq C\,\|x_0\|^\gamma.$$
\end{theorem}

The case $\gamma = 2$ is well-known (see, e.g. \cite{Bauschke96}); it is a fundamental ingredient of the asymptotic regularity property of the trajectories and it holds without the innate regularity assumption on the subspaces (but under the assumption that $\mathrm{limsup}~ \lambda_k < 2$). The strength of Theorem \ref{theo1} starts with $\gamma = 1$ because it goes beyond the norm convergence result known to hold for an innately regular $\sV$ 
and shows, in addition, that all trajectories fall into a ball within a proper subspace of convergent sequences in $\sH$, namely the space 
$$ \mathrm{bv}(\NN,\sH):=\left \{ f: \NN \to \sH : \sum_{n=0}^\infty \|f(n+1)-f(n)\| < \infty \right \}$$
of bounded variation functions from $\NN$ to $\sH$. This stronger sense of convergence is sometimes called {\em absolute convergence}, in analogy with the more common use of the term for series \cite{knopp1956infinite}. It simply amounts to saying that the displacements series
$$ x_0+\sum_{n=0}^\infty (x_{n+1} - x_n)$$
converges {\em absolutely} (to $\lim x_n$).

As $\gamma$ is decreased towards $0$, the strength of Theorem \ref{theo1} goes significantly beyond ensuring bounded total variation of the trajectories.
Quantifying the constant $C(\sV,\eta,\gamma)$ across all $0 < \gamma < \infty$, we also derive an effective bound on the distribution function of the norms of the displacements (see Proposition \ref{S-bound-lem}) and show that, despite the lack of possibility of establishing any effective convergence rate that holds uniformly for all trajectories, the $n$th largest displacement is bounded by $c \exp(-\rho n^{1/N})$ uniformly for all trajectories, i.e. the constants $c$ and $\rho$ only depend on $\sV$ and $\eta$ (see Theorem \ref{rearrangement-bound}). 

The paper is organized as follows: In Section \ref{angle}, we review the notion of angle between subspaces and its connection to the notion of innate regularity. Section \ref{geometry}, which is at the heart of the paper, is devoted to geometric properties of successive relaxed projections for innately regular subspaces which will be needed in our proof of Theorem \ref{theo1} given in Section \ref{proof-of-theo1}. Section \ref{statistics} is devoted to the statistical analysis of the displacements, and in particular, on the derivation of the aforementioned decay bound on the decreasing rearrangements of the displacements. 

\section{Angle between subspaces} \label{angle}

We start by recalling the notion of (complementary) angle between two subspaces introduced in \cite{friedrichs1937certain}; see \cite{Deutsch01} for a detailed discussion. Given two subspaces $V$ and $W$ of a Hilbert space $\sH$, the angle between $V$ and $W$ is defined to be the unique number $\varphi(V,W)\in[0,\frac{\pi}{2}]$ such that
\begin{equation}\label{angle-def}
\cos\varphi(V,W)=\sup\Big\{|\langle v,w\rangle|:v\in V\cap (V{\cap} W)^\perp, ~w\in W\cap (V{\cap} W)^\perp,~\|v\|\leq1~\mbox{and}~\|w\|\leq1\Big\}. 
\end{equation}

We note that there are some variations of this definition. Some authors restrict the test vectors $v$ and $w$ in \eqref{angle-def} to be of unit norm which requires the exclusion of the case of nested subspaces. Meanwhile, some authors allow for nested subspaces, but in this case separately set the angle between them to be $0$. Our choice for the definition of angle, as implied by \eqref{angle-def}, produces the value $\pi/2$ for nested subspaces (including the case $V=W$). This apparent discontinuity may seem counter-intuitive. However, there is also an intrinsic discontinuity in the problem we are considering in this paper: Both the limit of $x_n$ defined by \eqref{alg} and the associated total variation (the path length) $\sum \|x_{n+1} - x_n\|$ are discontinuous functions of $\sV$. This is most easily seen by considering alternating projections between two lines $\ell_1$ and $\ell_2$ in $\RR^2$ separated by an angle $\theta$. As we let $\theta \to 0^+$, $\lim x_n$ remains fixed at the origin while the path length blows up, but when $\ell_1=\ell_2$, $\lim x_n$ becomes the orthogonal projection of $x_0$ on $\ell_1$ and the total variation becomes finite. 

It follows from the discussion in the preceding paragraph and finite dimensional linear algebra that the angle between finite dimensional subspaces is always nonzero. However, the angle between infinite dimensional subspaces could be zero. In general, we have the following characterization of positive angle (see \cite[Proposition 5.16]{Bauschke96} and \cite[Theorem 9.35]{Deutsch01}):
For any two closed subspaces $V$ and $W$ in $\sH$, 
\begin{equation}\label{equiv-positive-angle}
\varphi(V,W)>0\quad\iff\quad V^\perp+W^\perp~\mbox{is closed}\quad\iff\quad V+W~\mbox{is closed}.
\end{equation}

\subsection{Innate regularity and its angular characterization} \label{angular_characterization}

When we have several subspaces in $\sV$, a very useful notion of angular separation for convergence of random projections turns out to be {\em innate regularity}. There are various levels of regularity applicable to general convex sets (see, e.g., \cite{bauschke1995norm,Bauschke96,bauschke2001projection}) but for subspaces they all boil down to a single notion also known as bounded linear regularity, which we will simply call regularity in this paper. Following \cite{bauschke1995norm}, a collection of subspaces $\sV=(V_1,\dots,V_n)$ is (boundedly linear) regular if there exists a constant $\kappa < \infty$ such that
\begin{equation}\label{linear_reg_dist}
d(x,V_1\cap \dots \cap V_N) \leq \kappa \max_i d(x,V_i) ~~~\mbox{ for all } x \in \sH,
\end{equation}
and {\em innately regular} if all of its (non-void) subcollections are regular. Here, $d(x,V)$ stands for the distance between $x \in \sH$ and the closed subspace $V$, also equal to $\|x - P_V x\|$ where $P_V$ is the orthogonal projection onto $V$. 

It is known that (see \cite[Theorem 5.19]{Bauschke96}) $\sV$ is regular if and only if  
$V_1^\perp + \cdots + V_N^\perp$ is closed. Therefore, as noted in \cite[Fact 3.2]{bauschke1995norm}), 
\begin{equation}\label{innate_1}
\sV \mbox{ is innately regular } \iff \sum_{i\in I}V_i^\perp \mbox{  is closed for all } I\subset[N]. 
\end{equation}
Here we take the sum over the empty collection to be the trivial (zero) subspace. 

For any $I\subset [N]$, let us use the notation 
\begin{equation}\label{V_I}
V_I:=\bigcap\limits_{i\in I}V_i 
\end{equation}
where we take $V_\emptyset:=\sH$. We identify $V_i$ with $V_{\{i\}}$. Hence with \eqref{equiv-positive-angle} we have
\begin{equation}\label{innate_2}
\sV \mbox{ is innately regular } \iff \varphi(V_I,V_J)>0 ~~\mbox{ for all } I,J \subset [N].
\end{equation}

As a special, but very important case, we note the following observation:
\begin{proposition}
Suppose that for every $i \in [N]$, $V_i$ has finite dimension or co-dimension. Then $\sV$ is innately regular.
\end{proposition}

\begin{proof}
Let $I,J\subset[N]$. If $V_I$ and $V_J$ both have finite dimension, then $V_I+V_J$, also having finite dimension, is closed. Otherwise, either $V_I$ or $V_J$ has finite co-dimension. Then $V_I^\perp+V_J^\perp$ is closed since the sum of a closed subspace and a finite dimensional subspace is always closed (see \cite[Lemma 9.36]{Deutsch01}). In either case, \eqref{equiv-positive-angle} yields $\varphi(V_I,V_J)>0$.
\end{proof}

\subsection{Quantifying regularity by means of angle} 

Consider two closed subspaces $V$ and $W$ of $\sH$. Since the collection $(V,W)$ is regular if and only if $\varphi(V,W) > 0$, it is natural to ask how the parameter $\kappa$ in \eqref{linear_reg_dist} is related to the angle $\varphi(V,W)$. While this specific relation will not be needed in this paper, the answer has a simple form which we note in the next proposition. 

\begin{proposition} \label{angle-regularity}
For any two closed subspaces $V$ and $W$ of $\sH$,
\begin{equation}
  d(x,V\cap W)~\sin \varphi(V,W) \leq d(x,V) + d(x,W) ~~\mbox{ for all } x \in \sH.
\end{equation}
In other words, for $N=2$, the constant $\kappa$ in \eqref{linear_reg_dist} can be chosen to be $2/\sin \varphi(V_1,V_2)$.
\end{proposition}
\begin{proof}
Let $P_U$ denote the orthogonal projection operator onto an arbitrary closed subspace $U$ of $\sH$.
For any $x \in \sH$, let $u:=P_{(V\cap W)^\perp}x$. Noting the relation
$P_Vu = P_V(x - P_{V\cap W}x) = P_Vx - P_{V\cap W}x$, we observe that
$P_Vu \in V \cap (V\cap W)^\perp$. Similarly, we have $P_Wu \in W \cap (V\cap W)^\perp$. Hence, as a consequence of \eqref{angle-def}, we have 
$$\sin \varphi(V,W) \leq \sin \varphi(P_Vu,P_Wu) \leq \sin \varphi(u,P_Vu) + \sin \varphi (u,P_Wu),$$
where $\varphi(v,w):=\varphi(\RR v,\RR w)$ denotes the angle between the lines defined by $v$ and $w$, and satisfies the triangle inequality. 
We multiply both sides of this inequality by $\|u\|=d(x,V\cap W)$. Observing that  
$$\|u\| \sin \varphi(u,P_V u) = d(u,V)  = \|P_{V^\perp} P_{(V\cap W)^\perp}x\| = \|P_{V^\perp} x\| = d(x,V)$$
(and similarly that $\|u\| \sin \varphi (u, P_W u) = d(x,W)$) yields the desired result.
\end{proof}

\begin{remark}
In fact, for distinct closed subspaces $V$ and $W$, it can be shown that
\begin{equation}\label{ineq1}
\sin\varphi(V,W) = \inf_{\substack{x\in(V{\cap}W)^\perp\\\|x\|=1}}\d(x,V)+\d(x,W).
\end{equation}
\end{remark}

\section{Geometry and dynamics of successive relaxed projections}\label{geometry}

\subsection{Geometry of one relaxed projection}

Let $P_V$ be the orthogonal projection operator onto the closed subspace $V$ of $\sH$. As before, for any $\lambda\in[0,2]$, we define the relaxed projection of $x \in \sH$ by 
$P_{V,\lambda} x:=(1-\lambda)x+\lambda P_Vx$.
The following are elementary derivations:
\begin{enumerate}
 \item[(E1)] $x - P_{V,\lambda} x = \lambda (x - P_Vx)$ so that $\|x - P_{V,\lambda} x\| = \lambda \|x - P_Vx\|$,
 \item[(E2)] $P_{V,\lambda} x - P_Vx = (1-\lambda)(x - P_Vx) \perp V$ so that $\|P_{V,\lambda} x\|^2 = \|P_Vx\|^2 + (1-\lambda)^2\|x - P_Vx\|^2$ and
 \item[(E3)] $\|x\|^2 - \|P_{V,\lambda} x\|^2 = \lambda(2-\lambda)\|x - P_Vx\|^2$. 
\end{enumerate}

This last statement trivially implies that $P_{V,\lambda}$ is non-expansive (i.e. $\|P_{V,\lambda} x\| \leq \|x\|$ for all $x \in \sH$). But it says more: provided $\lambda \in (0,2)$, $P_{V,\lambda}$ is strictly contractive if $x$ is not near $V$. More precisely, defining the relative distance function $\theta_V: \sH\to [0,1]$ via
\begin{equation}\label{rho-def}
 \theta_V(x) := \frac{d(x,V)}{\|x\|} = \frac{\|x - P_V x\|}{\|x\|}, ~~ x \not=0,
 \mbox { and }~~ \theta_V(0):=0,
\end{equation}
we have, for any $\eps \in [0,1]$, 
\begin{equation}\label{frac-dec}
\theta_V(x) \geq\eps \iff
\|P_{V,\lambda} x\| \leq (1-\lambda(2-\lambda)\eps^2)^{1/2} \,\|x\|.
\end{equation}
Note that $\lambda(2-\lambda)\eps^2 > 0$ if and only if $\lambda \in (0,2)$ and $\eps > 0$.

The lemma below states that the relaxed projection with respect to $W$ does not increase the relative distance with respect to any subspace $V$ of $W$:
\begin{lemma}\label{prop1}
Let $V$ and $W$ be any two closed subspaces of $\sH$ such that $V \subset W$. Then for all $\lambda \in [0,2]$ and $x \in \sH$,
$$\theta_V(P_{W,\lambda}x) \leq \theta_V(x).$$
\end{lemma}

\begin{proof}
Note that $y:=P_{W,\lambda} x$ is a convex combination of $x$ and $P_{W,2}x$. Since $P_{W,2}x$ is the mirror image of $x$ with respect to $W$, we have $\|P_{W,2}x\| = \|x\|$. More generally, $P_VP_{W,2}x = P_Vx$ implies
$$d(P_{W,2}x,V) = \|(P_Wx-x) +(P_Wx- P_Vx)\| = \|(P_Wx-x) - (P_Wx- P_Vx)\| =
d(x,V).$$ 
(The second equality above uses the fact that $P_Wx-x$ is orthogonal to $P_Wx-P_Vx \in W$.) Hence, by convexity, we have $d(y,V) \leq d(x,V)$. Since $P_V y = P_V x$, this implies 
$\tan \varphi(y,P_Vy) \leq \tan \varphi(x,P_Vx)$
and therefore 
$\theta_V(y) = \sin \varphi(y,P_Vy) \leq \sin \varphi(x,P_Vx) = \theta_V(x)$.
\end{proof}

Combining Proposition \ref{angle-regularity} and Lemma \ref{prop1} (where $(V,W)$ is replaced by $(V\cap W, W)$) yields the following corollary:
\begin{corollary} \label{corr-update}
Let $V$ and $W$ be any two closed subspaces of $\sH$ such that $\varphi(V,W) > 0$. Then for all $\lambda \in [0,2]$ and $x \in \sH$,
$$\theta_{V\cap W}(P_{W,\lambda}x) \leq \theta_{V\cap W}(x)
\leq \kappa(V,W) \max\Big(\theta_V(x), \theta_W(x)\Big),
$$
where $\kappa(V,W):=2/\sin \varphi(V,W)$.
\end{corollary}

\subsection{Dynamics of successive relaxed projections}

We fix an innately regular collection $\sV = (V_1,\dots,V_N)$ and define
\begin{equation}\label{kappa_sV}
\kappa_* := \kappa_*(\sV) := \max_{I, J \subset [N]} \kappa(V_I,V_J) 
\end{equation}
where $\kappa(V,W)$ is defined in Corollary \ref{corr-update}. Note that $2 \leq \kappa_* < \infty$.

Now consider any sequence $(x_n)_0^\infty$ of iterates defined by \eqref{alg}, i.e. $x_{n+1}:=P_{i_n,\lambda_n}(x_n)$, $n \geq 0$. Let $I_{-1} := \emptyset$ and 
\begin{equation}\label{I_n}
I_n := \Big \{i_k : 0 \leq k \leq n \Big \}, ~~n \geq 0. 
\end{equation}
To ease our notation, we will denote $\theta_{V_I}$ by $\theta_I$ for $I \subset [N]$, and $\theta_{V_i}$ by $\theta_i$ for $i \in [N]$, as there will be no possibility of confusion.
The following lemma will be useful in our analysis.

\begin{lemma}\label{growth-lemma}
Let $V_I$, $\kappa_*$, and $I_n$ be defined as in \eqref{V_I}, \eqref{kappa_sV}, and \eqref{I_n}, respectively. We have
\begin{equation}\label{growth-bound}
\theta_{I_n}(x_{n+1}) \leq \kappa_*^{|I_n|} \max_{0 \leq k \leq n}
\theta_{i_k}(x_k).
\end{equation}
\end{lemma}

\begin{proof}
We begin by applying Corollary \ref{corr-update} for $V=V_{I_{n-1}}$, $W = V_{i_n}$, 
$\lambda = \lambda_n$, $x=x_n$. Note that $V_{I_{n-1}} \cap V_{i_n} = V_{I_{n}}$. Note also that $i_n \in I_{n-1}$ implies $I_n = I_{n-1}$. Hence,
\begin{equation}\label{one-step}
\theta_{I_{n}}(x_{n+1}) \leq
\left\{
\begin{array}{ll}
  \theta_{I_{n-1}}(x_n), & \mbox{ if } i_n \in I_{n-1}, \\
  \kappa_* \max(\theta_{I_{n-1}}(x_n),\theta_{i_n}(x_n))_, & \mbox{ regardless.}
\end{array}
\right.
\end{equation}
We can now prove \eqref{growth-bound} by induction. Since $\theta_{I_{-1}}(x_0) = \theta_H(x_0) = 0$, the bound \eqref{one-step} yields $\theta_{I_0}(x_1) \leq \kappa_* \theta_{i_0}(x_0)$. With $|I_0|=1$, the statement \eqref{growth-bound} for $n=0$ follows.

For the induction step, we assume
$$
%\begin{equation}\label{growth-bound-1}
\theta_{I_{n-1}}(x_{n}) \leq \kappa_*^{|I_{n-1}|} \max_{0 \leq k \leq n-1}
\theta_{i_k}(x_k)
%\end{equation}
$$
and inject this bound into \eqref{one-step}. The two cases are as follows:
\begin{itemize}
 \item If $i_n \in I_{n-1}$, then $|I_n| = |I_{n-1}|$, so using the first bound in \eqref{one-step} we get 
$$\theta_{I_{n}}(x_{n+1}) \leq \theta_{I_{n-1}}(x_n) \leq
\kappa_*^{|I_{n-1}|} \max_{0 \leq k \leq n-1}
\theta_{i_k}(x_k) \leq \kappa_*^{|I_{n}|} \max_{0 \leq k \leq n}
\theta_{i_k}(x_k).
$$
\item 
If $i_n \not\in I_{n-1}$, then $|I_n| = |I_{n-1}|+1$, so using  
the second bound in \eqref{one-step} we get 
$$\theta_{I_{n}}(x_{n+1}) \leq \kappa_* \max \Big(
\kappa_*^{|I_{n-1}|} \max_{0 \leq k \leq n-1} \theta_{i_k}(x_k) ,
\kappa_*^{|I_{n-1}|} \theta_{i_n}(x_n) \Big)
\leq \kappa_*^{|I_{n}|} \max_{0 \leq k \leq n}
\theta_{i_k}(x_k).
$$
\end{itemize}
This completes the induction step and the proof.
\end{proof}

Let us make two observations:
\begin{observation}\label{obs1}
$\|x_k - x_{k+1}\| = \|x_k - P_{i_k,\lambda_k}x_k\| = \lambda_k \|x_k - P_{i_k} x_k\| = 
\lambda_k \theta_{i_k}(x_k) \|x_k\|$.
\end{observation}

\begin{observation}\label{propos0}
For all $0 \leq m \leq n$, 
$$x_{m+1}-x_0 = \sum_{k=0}^m (x_{k+1}-x_k) \in V_{i_0}^\perp + \cdots + V_{i_m}^\perp \subset
(V_{i_0} \cap \cdots \cap V_{i_m})^\perp \subset
V_{I_n}^\perp.$$
\end{observation}

\begin{proposition}\label{propos2}
Let $n \geq 0$. Suppose, 
for some $\eps < \kappa_*^{-|I_n|}$, we have
\begin{equation}\label{mode2}
\theta_{i_k}(x_k) \leq \eps~~\mbox{ for all }~ 0 \leq k \leq n.
\end{equation}
Then either $x_0 = 0$ or else $x_0 \not\in V_{I_n}^\perp$. In particular, if $0 \not= x_0 \in V_I^\perp$ for some $I \supset I_n$, then $I \not= I_n$.
\end{proposition}

\begin{proof}
Lemma \ref{growth-lemma} immediately implies $\theta_{I_n}(x_{n+1}) \leq \kappa_*^{|I_n|} \eps < 1$. If $x_{n+1} \not= 0$, then this means $d(x_{n+1},V_{I_n}) < \|x_{n+1}\|$ so that $x_{n+1} \not\in V_{I_n}^\perp$. Because $x_{n+1}-x_0 \in V_{I_n}^\perp$ due to Observation \ref{propos0}, it follows that $x_0 \not\in V_{I_n}^\perp$. 
Meanwhile, note that $\eps< 1/2$ so that $\lambda_k \theta_{i_k}(x_k) < 1$ for all $0 \leq k \leq n$. Hence, if $x_{n+1} = 0$ then Observation \ref{obs1} yields $x_n  = 0$, and therefore we  recursively obtain $x_j = 0$ for {\em all} $0 \leq j \leq n+1$. In other words, $x_0 \not=0$ implies that $x_{n+1} \not=0$. This completes the proof.
\end{proof}

\section{Proof of Theorem \ref{theo1}} \label{proof-of-theo1}

Given any innately regular collection $\sV$ in $\sH$ and $\eta \in (0,1]$, let us define
$$\eps_*:= \eps_*(\sV) := \frac{1}{2}\kappa_*^{-N}$$
where $\kappa_*$ is defined in \eqref{kappa_sV}, and
$$\beta_* := \beta_*(\sV,\eta):=\left(1-\eta(2{-}\eta)\eps_*^2\right)^{1/2}.$$
Also, for any $\gamma > 0$, let us define the increasing sequence of positive numbers  
$C_\ell:=C_\ell(\sV,\eta,\gamma)$ for $1 \leq \ell \leq N$ via
$$C_{\ell+1} := C_\ell + \frac{C_\ell+(2{-}\eta)^\gamma}{1-\beta_*^\gamma},~~~~
1 \leq l < N,~~\mbox{ where}~~C_1:=\frac{(2{-}\eta)^\gamma}{1-(1{-}\eta)^\gamma}.$$

For each $\ell=1,\dots,N$, let $\mathbf{P}(\ell)$ be the following statement:

{\em For all integers $q \geq p \geq 0$, if the control sequence $(i_k)_p^q$ takes at most $\ell$ distinct values in $[N]$ and the relaxation sequence $(\lambda_k)_p^q$ is in $[\eta,2-\eta]$, then any trajectory $(x_n)_p^{q+1}$ defined by 
$$x_{k+1} := P_{i_k,\lambda_k}(x_k),~~~k=p,\dots,q$$
satisfies
\begin{equation}\label{new-theo-bound}
\sum_{k= p}^{q}\|x_{k+1}-x_k\|^\gamma \leq C_{\ell}\,\|x_p\|^\gamma.
\end{equation}
}
 
We will prove $\mathbf{P}(\ell)$ by induction on $\ell$. 

Before we start the proof, consider the following point which is independent of $\ell$:
For any range of integers $[p,q]$ and trajectory $(x_k)_p^{q+1}$ with control sequence $(i_k)_p^q$, if we define 
\begin{equation}\label{un}
y_k:=x_k-P_{I_{p,q}} x_p,\qquad k\in[p,q{+1}],
\end{equation}
where $I_{p,q}:= \big \{i_k : k \in [p,q] \big \}$ (and 
$P_{I_{p,q}}$ is short for $P_{V_{I_{p,q}}}$), then 
\begin{itemize}
 \item[(i)] since $(y_k)$ is a translation of $(x_k)$, we have $y_{k+1}-y_k = x_{k+1}-x_k$ for all $k \in [p,q]$,
 \item[(ii)] since $P_{I_{p,q}} x_p\in V_{I_{p,q}}\subset V_{i_k}$ for all $k\in[p,q]$, we have $P_{i_k,\lambda_k}P_{I_{p,q}} x_p=P_{I_{p,q}} x_p$ so that 
\begin{equation}\label{alg-y}
y_{k+1} = P_{i_k,\lambda_k}y_k,\qquad k\in[p,q],
\end{equation}
and
\item[(iii)] since  $y_p\in V_{I_{p,q}}^\perp$, we have $y_k\in V_{I_{p,q}}^\perp$ for all $k\in[p,q{+}1]$ as a result of Observation \ref{propos0}.
\end{itemize}

We now start the proof with the base case $\ell=1$, which means that for some $i \in [N]$, we have $P_{i_k} = P_i$ for all $k \in [p,q]$ (in other words $I_{p,q} = \{i\}$).
Let $(y_k)$ be defined as in \eqref{un}. Noting that $P_{i_k} y_k = 0$ for all $k \in [p,q{+}1]$, the relation \eqref{alg-y} implies via \eqref{relaxedPi} that $y_{k+1} = (1-\lambda_k) y_k$ for all $k \in [p,q]$. Since $|1-\lambda_k|\leq 1-\eta$, it then follows that 
$$\|y_k\|\leq(1-\eta)^{k-p}\|y_p\|\leq(1-\eta)^{k-p}\|x_p\| ~~\mbox{ for all } k \in [p,q{+}1],$$
so that 
$$\|x_{k+1}-x_k\| = \|y_{k+1} - y_k\| \leq (2-\eta)\|y_k\| \leq (2-\eta)(1-\eta)^{k-p}\|x_p\|~~\mbox{ for all } k \in [p,q].$$
Summing the bound raised to the power $\gamma$ then yields \eqref{new-theo-bound}. Hence we have shown $\mathbf{P}(1)$.

For the induction step, consider any $\ell < N$ and assume that $\mathbf{P}(\ell)$ holds. The case $N=1$ is vacuous, so we may assume $N\geq 2$. 

We will deduce the truth of $\mathbf{P}(\ell{+}1)$. 
Let $(i_k)_p^q$ take at most $\ell+1$ distinct values in $[N]$, i.e. $|I_{p,q}| \leq \ell+1$. Given an associated trajectory $(x_k)_p^q$, again let $(y_k)_p^q$ be defined as in \eqref{un}. Since $y_p \in V_{I_{p,q}}^\perp$, Proposition \ref{propos2} with $\eps=\eps_*$ shows that either $y_{p} =0$ (and we are done because then $y_k = 0$ for all $k \in [p,q]$) or else $\theta_{i_k}(y_k) > \eps_*$ for some $k \in [p,q]$. In this case,
let us enumerate the set $\big \{k \in [p,q] : \theta_{i_k}(y_k) > \eps_* \big \}$ as an increasing sequence $r_1 < \cdots < r_L$. This results in a segmentation of $[p,q]$ in the form
$$
[p_0,q_0] \cup \{r_1\} \cup [p_1,q_1] \cup \cdots \cup \{r_L\} \cup [p_L,q_L]
$$
where $p_0 := p$, $q_L := q$ and for all $j=1,\dots,L$ we have $p_j := r_j + 1$ and $q_{j-1} := r_j -1$. (With this notation, we allow for the possibility that $q_j = p_j-1$ which simply means that $[p_j,q_j] = \emptyset$.) Let us also define $r_0 := p$.

For each $j=0,\dots,L$, we have $\theta_{i_k}(y_k) \leq \eps_*$ for all $k \in [p_j,q_j]$, so Proposition \ref{propos2} implies that 
either $y_{p_j}=0$ (in which case $y_k = 0$ for all $k \geq p_j$) or $I_{p_j,q_j}$ must be a {\em proper} subset of $I_{p,q}$ so that $|I_{p_j,q_j}| \leq \ell$. In this case, $\mathbf{P}(\ell)$ yields
\begin{equation}\label{bound-sub-seg-2}
\sum_{k=p_j}^{q_j} \|y_{k+1}-y_k \|^\gamma \leq C_\ell \|y_{p_j}\|^\gamma
\leq C_\ell \|y_{r_j}\|^\gamma, ~~~j=0,\dots,L.
\end{equation}
Meanwhile, \eqref{frac-dec} implies
$$\|y_{r_j+1}\| \leq \beta_* \|y_{r_j}\|, ~~~j=1,\dots,L.$$
Due to the fact that $\|y_k\|$ is a monotonically decreasing sequence, this results in the decay bound
$$ \|y_{r_j}\| \leq \beta_*^{j-1} \|y_{r_1}\| \leq \beta_*^{j-1} \|y_{r_0}\|, ~~~j=1,\dots,L,$$
so that 
$$ \|y_{r_j+1} - y_{r_j}\| \leq (2{-}\eta) \|y_{r_j}\| \leq (2{-}\eta) \beta_*^{j-1} \|y_{r_0}\|, ~~~j=1,\dots,L.$$
Combined with \eqref{bound-sub-seg-2}, we obtain 
\begin{eqnarray}
\sum_{k=p}^{q} \|y_{k+1}-y_k \|^\gamma 
& \leq &
C_\ell \sum_{j=0}^L \|y_{r_j}\|^\gamma + (2{-}\eta)^\gamma \sum_{j=1}^L \|y_{r_j}\|^\gamma \nonumber \\
& \leq & 
\Big (C_\ell + \frac{C_\ell+(2{-}\eta)^\gamma}{1-\beta_*^\gamma}\Big) \|y_p\|^\gamma
~=~ C_{\ell+1} \|y_p\|^\gamma
\end{eqnarray}
so that 
$$ \sum_{k=p}^{q} \|x_{k+1}-x_k \|^\gamma 
= \sum_{k=p}^{q} \|y_{k+1}-y_k \|^\gamma 
\leq C_{\ell+1} \|y_p\|^\gamma \leq C_{\ell+1} \|x_p\|^\gamma.
$$
Hence $\mathbf{P}(\ell{+}1)$ holds, completing the induction step. 
Since $p$ and $q$ are arbitrary, 
Theorem \ref{theo1} readily follows from $\mathbf{P}(N)$ 
with $C(\sV,\eta,\gamma):=C_N$.
\qed

\begin{remark}\label{C-bound-rem}
Since $(2{-}\eta)^\gamma \leq C_1 \leq C_\ell$, we have $C_{\ell+1} \leq 3C_\ell/(1-\beta_*^\gamma)$, and therefore 
\begin{equation}\label{C-bound-1}
 C(\sV,\eta,\gamma) := C_N \leq
\left(\frac{3}{1-\beta_*^\gamma}\right)^{N-1}\frac{(2-\eta)^\gamma}{1-(1-\eta)^\gamma}.
\end{equation} 
\end{remark}

\section{Statistics of displacements via moment bounds} \label{statistics}

While $x_n$ can be arranged to converge to its limit arbitrarily slowly, the moment bounds of Theorem \ref{theo1} place strong restrictions on the number of displacements exceeding any given value. In this section we will quantify this proposition.

Let us fix $\sV$ and $\eta\in(0,1]$ according to Theorem \ref{theo1} and consider any trajectory $(x_n)_0^\infty$ where $x_0 \not=0$. Since
$\|x_{n+1}-x_n\|\leq (2-\eta) \|x_n\| \leq (2-\eta) \|x_0\|$, let us define
$$ \delta_n := \frac{\|x_{n+1} - x_n\|}{(2{-}\eta)\|x_0\|},~~~n \in \NN,$$
as a normalized measure of the displacements. For any $\tau \in [0,1]$, let us also define
\begin{equation}\label{def-S}
S(\tau) :=  |\Lambda_\tau|, ~~\mbox{ where }~
\Lambda_\tau := \Big \{ n \in \NN: \delta_n \geq \tau \Big \}. 
\end{equation}
The next proposition shows that $S(\tau) = O(|\log \tau|^N)$ as $\tau \to 0$.
\begin{proposition}\label{S-bound-lem}
Assume the hypothesis of Theorem \ref{theo1}. Let $\beta_*$ be defined as in Section \ref{proof-of-theo1} and $S(\tau)$ as above where $x_0 \not=0$. Then for all $\tau \in (0,1]$ we have
\begin{equation} \label{cheb-bound-3}
 S(\tau) \leq 
3^{N-1} e^{N} \left(1 + \frac{\log \tau}{N \log \beta_*}\right)^{N-1} 
\left(1 + \frac{\log \tau}{N\log (1-\eta)}\right).
\end{equation}
In particular,
\begin{equation} \label{cheb-bound-4}
 S(\tau) <
9^{N} \left(1 + \frac{\log \tau}{N\log \beta_*}\right)^{N} .
\end{equation}
\end{proposition}

\begin{proof}
Let $\tau \in (0,1]$ be arbitrary. With Theorem \ref{theo1} we have
$$ (2{-}\eta)^{-\gamma} C(\sV,\eta,\gamma) \geq 
\sum_{n=0}^\infty \frac{\|x_{n+1}-x_n\|^\gamma}{(2{-}\eta)^\gamma \|x_0\|^\gamma} = \sum_{n=0}^\infty  \delta_n^\gamma \geq \sum_{n\in \Lambda_\tau}  \delta_n^\gamma \geq \tau ^\gamma S(\tau).$$
This inequality holds for all $0< \gamma < \infty$, so 
\begin{eqnarray}
 S(\tau) 
 & \leq &
 \inf_{0 < \gamma < \infty}  \tau^{-\gamma}(2{-}\eta)^{-\gamma}C(\sV,\eta,\gamma) \nonumber  \\
 & \leq & 
 \inf_{0 < \gamma < \infty} \tau^{-\gamma}
\left(\frac{3}{1-\beta_*^\gamma}\right)^{N-1}\frac{1}{1-(1{-}\eta)^\gamma}, \label{cheb-bound}
\end{eqnarray}
where in the last step we have used the explicit bound derived in Remark \ref{C-bound-rem}.

To ease our computation, we slightly relax the upper bound.
Note that, for any $0 < r < 1$,
$$ \frac{1}{1-r^\gamma} = 1 + \frac{1}{r^{-\gamma} - 1} 
\leq 1 + \frac{1}{\log r^{-\gamma}}
= 1 + \frac{1}{\gamma\log r^{-1}},$$
so that 
\begin{equation}\label{cheb-bound-2}
S(\tau) \leq 3^{N-1} \inf_{0 < \gamma < \infty}  \tau^{-\gamma}
\left(1 + \frac{1}{\gamma \log \beta_*^{-1}}\right)^{N-1} \left(1 + \frac{1}{\gamma \log (1{-}\eta)^{-1}}\right).
\end{equation}

For $\tau = 1$, we get $S(1) \leq 3^{N-1}$, hence \eqref{cheb-bound-3}. (In fact, it can be shown that $S(1) \leq 1$.)

Let us assume $\tau \in (0,1)$.
Noting that $\tau^{-\gamma}\gamma^{-N}$ is minimized at $\gamma_\tau := N/ \log \tau^{-1}$, we may set $\gamma = \gamma_\tau$ in \eqref{cheb-bound-2}. The desired bound of \eqref{cheb-bound-3} follows immediately once we observe $\tau^{-1/\log \tau^{-1}} = e$. Then \eqref{cheb-bound-4} follows from the simple relations 
$3^{N-1}e^N < 9^N$ and $ 1 > \beta_* > (1-\eta(2-\eta))^{1/2} = 1-\eta$.
\end{proof}

\begin{remark}
We note that the distinction between \eqref{cheb-bound-3} and \eqref{cheb-bound-4} may be negligible for small values of $\eta$, but in the case of no relaxation ($\eta=1$) \eqref{cheb-bound-4} carries an extra factor of $\log \tau$, and is therefore suboptimal.
\end{remark}

As an immediate application of this proposition, we will derive an explicit decay estimate for the decreasing rearrangement of $(\delta_n)_0^\infty$ which we denote by $(\delta^*_n)_0^\infty$. Recall that this is the (unique) sequence
$$ \delta^*_0 \geq  \delta^*_1 \geq \cdots $$
satisfying
$\delta^*_n = \delta_{\pi(n)}$ for some bijection $\pi:\NN \to \NN$.

\begin{theorem} \label{rearrangement-bound}
Assume the hypothesis of Proposition \ref{S-bound-lem}. Then
\begin{equation} \label{delta-star-decay}
\delta^*_n < c_* \exp(-\rho_* n^{1/N}) ~~~\mbox{ for all }~~ n\geq 0,
\end{equation}
where 
$\rho_* := \frac{N}{9} \log\beta_*^{-1} > 0 $ and $c_*:= \beta_*^{-N}$.
\end{theorem}

\begin{proof}
The result holds trivially when $\delta^*_n = 0$, so it suffices to consider the nonzero values only. Note that 
$$ S(\delta^*_n) = \left|\Big \{ k \in \NN: \delta_k \geq \delta^*_n \Big \}\right| = \left|\Big \{ k \in \NN: \delta^*_k \geq \delta^*_n \Big \}\right|
\geq n+1 $$
so that $n < S(\delta^*_n)$ which implies, when combined with Proposition \ref{S-bound-lem},
$$ n < 9^N 
\left(1 + \frac{\log \delta^*_n}{N\log \beta_*}\right)^{N}.$$
The desired bound \eqref{delta-star-decay} then easily follows from this inequality by solving for $\delta^*_n$.
\end{proof}

\paragraph{Acknowledgment.} The authors thank Heinz Bauschke for a helpful correspondence and Halyun Jeong for an insightful inquiry which triggered a notable strengthening of our result.

\bibliographystyle{alpha}
\bibliography{reference}{}

\begin{thebibliography}{Com96}

\bibitem[AA65]{amemiya1965convergence}
I.~Amemiya and T.~Ando.
\newblock Convergence of random products of contractions in {H}ilbert space.
\newblock {\em Acta Sci. Math.(Szeged)}, 26(3-4):239--244, 1965.

\bibitem[AC89]{aharoni1989block}
Ron Aharoni and Yair Censor.
\newblock Block-iterative projection methods for parallel computation of
  solutions to convex feasibility problems.
\newblock {\em Linear Algebra and Its Applications}, 120:165--175, 1989.

\bibitem[Bau95]{bauschke1995norm}
Heinz~H. Bauschke.
\newblock A norm convergence result on random products of relaxed projections
  in {H}ilbert space.
\newblock {\em Transactions of the American Mathematical Society},
  347(4):1365--1373, 1995.

\bibitem[Bau01]{bauschke2001projection}
Heinz~H. Bauschke.
\newblock Projection algorithms: results and open problems.
\newblock In {\em Studies in Computational Mathematics}, volume~8, pages
  11--22. Elsevier, 2001.

\bibitem[BB96]{Bauschke96}
Heinz~H. Bauschke and Jonathan~M. Borwein.
\newblock On projection algorithms for solving convex feasibility problems.
\newblock {\em SIAM Rev.}, 38(3):367--426, 1996.

\bibitem[BC17]{bauschke2017convex}
Heinz~H Bauschke and Patrick~L Combettes.
\newblock Convex analysis and monotone operator theory in hilbert spaces.
\newblock 2017.

\bibitem[Ceg12]{cegielski2012iterative}
Andrzej Cegielski.
\newblock {\em Iterative methods for fixed point problems in {H}ilbert spaces},
  volume 2057.
\newblock Springer, 2012.

\bibitem[Com96]{combettes1996convex}
PL~Combettes.
\newblock The convex feasibility problem in image recovery.
\newblock In {\em Advances in imaging and electron physics}, volume~95, pages
  155--270. Elsevier, 1996.

\bibitem[Deu01]{Deutsch01}
Frank Deutsch.
\newblock {\em Best approximation in inner product spaces}, volume~7 of {\em
  CMS Books in Mathematics/Ouvrages de Math\'{e}matiques de la SMC}.
\newblock Springer-Verlag, New York, 2001.

\bibitem[Fri37]{friedrichs1937certain}
Kurt Friedrichs.
\newblock On certain inequalities and characteristic value problems for
  analytic functions and for functions of two variables.
\newblock {\em Transactions of the American Mathematical Society},
  41(3):321--364, 1937.

\bibitem[Hal62]{halperin1962product}
Israel Halperin.
\newblock The product of projection operators.
\newblock {\em Acta Sci. Math.(Szeged)}, 23(1):96--99, 1962.

\bibitem[Kac37]{Kaczmarz37}
S.~Kaczmarz.
\newblock Angen\"aherte {A}ufl\"osung von {S}ystemen linearer {G}leichungen.
\newblock {\em Bulletin de l'Acad\'emie des Sciences de Pologne}, A35:355--357,
  1937.

\bibitem[KM14]{kopecka2014product}
Eva Kopeck{\'a} and Vladim{\'\i}r M{\"u}ller.
\newblock A product of three projections.
\newblock {\em Studia Mathematica}, 2(223):175--186, 2014.

\bibitem[Kno56]{knopp1956infinite}
Konrad Knopp.
\newblock {\em Infinite sequences and series}.
\newblock Courier Corporation, 1956.

\bibitem[KP17]{kopecka2017strange}
Eva Kopeck{\'a} and Adam Paszkiewicz.
\newblock Strange products of projections.
\newblock {\em Israel Journal of Mathematics}, 219(1):271--286, 2017.

\bibitem[Opp18]{oppenheim2018angle}
Izhar Oppenheim.
\newblock Angle criteria for uniform convergence of averaged projections and
  cyclic or random products of projections.
\newblock {\em Israel Journal of Mathematics}, 223(1):343--362, 2018.

\bibitem[Pr{\'a}60]{prager1960uber}
M~Pr{\'a}ger.
\newblock Uber ein {K}onvergenzprinzip im {H}ilbertschen {R}aum.
\newblock {\em Czechoslovak Math. J}, 10(85):271--272, 1960.

\bibitem[PRZ12]{pustylnik2012convergence}
Evgeniy Pustylnik, Simeon Reich, and Alexander~J Zaslavski.
\newblock Convergence of non-periodic infinite products of orthogonal
  projections and nonexpansive operators in hilbert space.
\newblock {\em Journal of Approximation Theory}, 164(5):611--624, 2012.

\bibitem[Sak95]{sakai1995strong}
Makoto Sakai.
\newblock Strong convergence of infinite products of orthogonal projections in
  hilbert space.
\newblock {\em Applicable Analysis}, 59(1-4):109--120, 1995.

\bibitem[vN50]{von1950functional}
John von Neumann.
\newblock Functional operators. vol. ii. the geometry of orthogonal spaces,
  volume 22, annals of math studies.
\newblock {\em Studies. Princeton University Press}, 1950.

\end{thebibliography}

\end{document}